\pdfoutput=1
\documentclass[a4paper,12pt,reqno]{amsart}
\usepackage{typearea}
\typearea{13}

\usepackage[utf8]{inputenc}
\usepackage[UKenglish]{babel}
\usepackage[T1]{fontenc}

\usepackage{amsmath}
\usepackage{amsthm}
\usepackage{amsfonts}
\usepackage{amssymb}
\usepackage{enumitem}

\usepackage{dsfont}
\usepackage{ifthen}
\usepackage{mathtools}
\usepackage{mathrsfs}
\usepackage{etoolbox}
\usepackage{xparse}

\usepackage{microtype}
\usepackage{url}
\usepackage[plainpages=false,colorlinks=false,unicode]{hyperref}

\mathtoolsset{centercolon}

\theoremstyle{plain}
\newtheorem{thm}{Theorem}[section]
\newtheorem{prop}[thm]{Proposition}
\newtheorem{cor}[thm]{Corollary}
\newtheorem{lem}[thm]{Lemma}
\theoremstyle{definition}

\newtheorem{rem}[thm]{Remark}
\newtheorem*{rem*}{Remark}
\newtheorem{definition}[thm]{Definition}

\newcommand{\mathdcl}[1]{{\textup{#1}}}
\newcommand{\cpt}{\mathdcl{c}}
\newcommand{\nor}{\mathdcl{n}}

\newcommand{\RR}{\mathbb{R}}
\newcommand{\NN}{\mathbb{N}}
\newcommand{\CC}{\mathbb{C}}

\newcommand{\one}{\mathds{1}}
\newcommand{\eps}{\varepsilon}
\renewcommand{\phi}{\varphi}

\newcommand{\Linop}{\mathcal{L}}
\newcommand{\conj}[1]{\overline{#1}}
\newcommand{\clos}[1]{\overline{#1}}
\newcommand{\form}[1]{{\mathfrak{#1}}}
\newcommand{\restrict}[2]{\ensuremath{#1\raisebox{-0.4ex}{$|$}\strut_{#2}}}

\renewcommand{\Re}{\operatorname{Re}}
\renewcommand{\Im}{\operatorname{Im}}

\newcommand{\leconst}{\lesssim}

\newcommand{\dx}[1][x]{\,\mathrm{d}#1}
\newcommand{\Hm}[1][d-1]{\mathcal{H}^{#1}}

\DeclarePairedDelimiter\norm{\lVert}{\rVert}
\DeclarePairedDelimiter\abs{\lvert}{\rvert}

\makeatletter
\let\old@norm\norm
\@ifpackageloaded{xparse}{
\DeclareDocumentCommand{\smartnorm}{soG{\cdot}}{%
  \IfBooleanTF{#1}{\old@norm{#3}}%
  {%
    \IfValueTF{#2}{\old@norm[#2]{#3}}{\old@norm*{#3}}%
  }%
}
\def\norm{\smartnorm}
}{}
\makeatother

\newcommand{\scalar}[3][auto]{{
\ifthenelse{\equal{#1}{auto}}{\left(#2\mkern3mu{\mid}\mkern3mu #3\right)}{}
\ifthenelse{\equal{#1}{b}}{\bigl(#2\mkern3mu{\mid}\mkern3mu #3\bigr)}{}
\ifthenelse{\equal{#1}{B}}{\Bigl(#2\mkern3mu{\bigm|}\mkern3mu #3\Bigr)}{}
}}

\newcommand{\emphdef}[1]{\textbf{\boldmath #1\unboldmath}}

\newlist{romanenum}{enumerate}{1}
\setlist[romanenum]{label=\textup{(\roman*)},ref=\textup{(\roman*)},itemsep=0em,topsep=1ex}

\newlist{alenum}{enumerate}{1}
\setlist[alenum]{label=\textup{(\alph*)},ref=\textup{(\alph*)},itemsep=0em,topsep=1ex}

\newlist{parenum}{enumerate}{1}
\setlist[parenum]{label=\textup{\arabic*.},ref=\textup{\arabic*.},wide,noitemsep}

\makeatletter
\newtoks\protect@toks
\def\pdef#1{\protect@toks=\expandafter{\the\protect@toks
 \pdef@#1}%
 \def#1}
\def\pdef@#1{\def#1{\protect#1}}
\newdimen\mex
\pdef\niv{\mathrel{\hbox{\hglue .2\mex
  \vrule \@height 1.33\mex \@width .06\mex
  \vrule \@height .06\mex \@width 1\mex
  \hglue .5\mex}}}

\makeatother

\newcommand{\ident}{\equiv}
\DeclareMathOperator{\sign}{sign}
\DeclareMathOperator{\Tr}{Tr}

\newcommand{\rmN}{\mathdcl{N}}
\newcommand{\rmW}{\mathdcl{W}}
\newcommand{\rmC}{\mathdcl{C}}
\newcommand{\fa}{\form{a}}
\newcommand{\fb}{\form{b}}

\renewcommand{\jmath}{j}
\let\oldmarginpar\marginpar
\renewcommand{\marginpar}[1]{\oldmarginpar{\raggedright\tiny #1}}

\let\oldbibitem\bibitem
\renewcommand{\bibitem}[2][]{\oldbibitem{#2}}

\makeatletter
\@namedef{subjclassname@2020}{%
  \textup{2020} Mathematics Subject Classification}
\makeatother

\title[Wentzell Laplacian via forms and approximative trace]{The Wentzell Laplacian via forms and the approximative trace}

\dedicatory{Dedicated to Jerry Goldstein on the occasion of his 80th birthday.}

\author[W. Arendt]{\sc Wolfgang Arendt}
\address{Wolfgang Arendt\\Institute of Applied Analysis\\Ulm University\\89069 Ulm\\Germany}
\email{wolfgang.arendt@uni-ulm.de}
\author[M. Sauter]{\sc Manfred Sauter}
\address{Manfred Sauter\\Institute of Applied Analysis\\Ulm University\\89069 Ulm\\Germany}
\email{manfred.sauter@uni-ulm.de}

\keywords{Wentzell boundary conditions, form methods, sectorial forms, approximative trace, irregular domains, semigroups, continuous kernel}
\subjclass[2020]{Primary: 35J05; Secondary: 47D60, 31C25, 46E35}

\begin{document}
\begin{abstract}
We use form methods to define suitable realisations of the Laplacian 
on a domain $\Omega$ with Wentzell boundary conditions, i.e.\ such that $\partial_\nor u + \beta u + \Delta u = 0$ holds in a suitable sense on the boundary of $\Omega$. For those realisations, we study their semigroup generation properties.
Using the approximative trace, 
we give a unified treatment that in part allows irregular and even fractal domains.
Moreover, we admit $\beta$ to be merely essentially bounded and complex-valued. If the domain is Lipschitz, we obtain a kernel continuous up to the boundary.
\end{abstract}

\begingroup
\renewcommand{\MakeUppercase}[1]{#1}
\maketitle
\endgroup

\section{Introduction}

In 2002 A.~Favini, G.R.~Goldstein, J.A.~Goldstein and S.~Romanelli~\cite{FGGR02} published an article in the Journal of Evolution Equations which turned out to be hugely influential and is among the most cited in this journal.
It treats parabolic equations with Wentzell boundary conditions on a suitably regular domain $\Omega$, whose boundary we will denote by $\Gamma$. One reason for the great popularity of the subject is certainly its importance for models: these boundary conditions govern a dynamic behaviour at the boundary.
Another reason is a challenge of analytical nature: in the $L^p$-context the natural space is $L^p(\Omega)\oplus L^p(\Gamma)$, as is shown in that paper.
So the behaviour in the interior and at the boundary is suitably decoupled. This technique was previously employed by H.~Amann and J.~Escher~\cite{AE96}.
The strategy in~\cite{FGGR02} is to prove m-dissipativity and to use the Hille--Yosida~theorem.
We refer to A.~Favini, G.R.~Goldstein, J.A.~Goldstein, E.~Obrecht and S.~Romanelli~\cite{FGGOR2016} for more recent contributions in this direction.

Shortly after the publication of~\cite{FGGR02}, a different approach via forms in $L^2(\Omega)\oplus L^2(\Gamma)$ was presented in~\cite{AMPR03} and utilised by D.~Mugnolo and S.~Romanelli in~\cite{MR06}.
Forms have the great advantage that the range condition comes for free.
The disadvantage is that the domain of the generator is not known precisely; it is the form domain which is given.
In~\cite{AMPR03} a semigroup is obtained on $L^2(\Omega)\oplus L^2(\Gamma)$ for Lipschitz domains.
It leaves the space $\{(u,u_\Gamma): u\in H^1(\Omega)\}$ invariant and one obtains a semigroup on $H^1(\Omega)$ that properly realises the Wentzell boundary conditions in sense of traces, see also~\cite{FGGOER03}.
Another most natural choice is to consider the space $C(\clos{\Omega})$. In both~\cite{FGGR02} and~\cite{AMPR03} this could be accomplished for $C^{2+\eps}$-domains in virtue of Schauder estimates.

M.~Warma~\cite{War06} and R.~Nittka~\cite{Nittka2010:thesis,Nit11} proved the necessary elliptic estimates to realise Wentzell--Robin boundary conditions on the space $C(\clos{\Omega})$ for a Lipschitz domain. Nittka's result even holds for arbitrary elliptic operators with real, bounded, measurable coefficients.

More recently, second-order evolution equations with Wentzell boundary conditions have been studied on the Koch snowflake domain in $\RR^2$ and certain other fractal extension domains; we refer to M.R.~Lancia and P.~Vernole~\cite{LV14} and M.~Hinz, M.R.~Lancia, A.~Teplyaev and P.~Vernole~\cite{HLTV18}.

The purpose of the present paper is as follows. We study the Laplacian $\Delta$ with Wentzell boundary conditions
\[
    \partial_\nor u + \beta\restrict{u}{\Gamma} + \restrict{(\Delta u)}{\Gamma} = 0,
\]
where $\Omega$ is a bounded open set with boundary $\Gamma=\partial\Omega$, and $\Gamma$ is equipped with a suitable finite Borel measure $\sigma$. In particular, if the $(d-1)$-dimensional Hausdorff measure $\Hm$ of $\Gamma$ is finite (which is the case for Lipschitz domains), then $\sigma=\Hm\niv\Gamma$ is a natural and suitable choice.
We make use of the approximative trace and Green's formula to realise the operator first on $L^2(\Omega)\oplus L^2(\Gamma)$, and then on a subspace $V$ of $H^1(\Omega)$ on which the boundary conditions are attained in the sense of traces.
This approach allows us to also consider irregular domains $\Omega$, including fractal domains for a suitable choice of $\sigma$.

In the special case of Lipschitz domains, however, we refine the corresponding results of~\cite{AMPR03} and recover those in~\cite{War06,Nittka2010:thesis,Nit11}.
In particular, we show the existence of a $C_0$-semigroup on $C(\clos{\Omega})$ if $\Omega$ has Lipschitz boundary. There is one further difference to~\cite{AMPR03,War06,Nittka2010:thesis,Nit11}. We admit $\beta\in L^\infty(\Gamma)$ to be complex-valued.
Thus the semigroup on $L^2(\Omega)\oplus L^2(\Gamma)$ is no longer submarkovian and not even positive.
We use the Beurling--Deny--Ouhabaz criterion to show that after rescaling this semigroup is $L^\infty$-contractive, without any positivity assumption, and then show the invariance of $F=\{(u,\restrict{u}{\Gamma}) : u\in C(\clos{\Omega})\}$ if $\Omega$ has Lipschitz boundary. In order to prove that the restriction of the semigroup to $F$ is strongly continuous, we argue in a similar way as Nittka in~\cite[Section~5.2]{Nittka2010:thesis} or~\cite[Section~4.2]{Nit11}.
Furthermore, if $\Omega$ is Lipschitz we show that the semigroup on $C(\clos{\Omega})$ has a continuous kernel on $\clos{\Omega}$ in a suitable sense.

\section{Form method for the incomplete case}\label{sec:fm-incompl}

Let $H$ be a complex Hilbert space.
We describe a method to associate in a natural way a holomorphic $C_0$-semigroup on $H$ to any sectorial form in $H$.
A \emphdef{sectorial form} in $H$ is a pair $(\fa,j)$ of a sesquilinear form 
\[
    \fa\colon D(\fa)\times D(\fa)\to\CC,
\]
where $D(\fa)$ is a complex vector space, and of a linear map $j\colon D(\fa)\to H$ with dense image such that there exist $\omega\in\RR$ and $\theta\in[0,\frac{\pi}{2})$ with
\[
    \fa(u,u) + \omega\norm{j(u)}_H^2 \in \Sigma_\theta
\]
for all $u\in D(\fa)$; here
\[
    \Sigma_\theta = \{ re^{i\alpha} : r\ge 0,\ \abs{\alpha}\le\theta\}
\]
denotes a closed complex sector that opens in direction of the positive real axis. 

To a sectorial form $(\fa,j)$ we associate the operator $A$ in $H$ defined as follows:  For all $x,y\in H$ one has $x\in D(A)$ and $Ax=y$ if and only if there exists a sequence $(u_n)_{n\in\NN}$ in $D(\fa)$ such that
\begin{romanenum}
\item $\displaystyle \lim_{n\to\infty} j(u_n)=x$ in $H$,
\item $\displaystyle \lim_{n,m\to\infty} \fa(u_n-u_m,u_n-u_m)=0$, and
\item $\displaystyle \lim_{n\to\infty} \fa(u_n,v) = \scalar{y}{j(v)}_H$ for all $v\in D(\fa)$.
\end{romanenum}

This operator $A$ is well-defined (i.e.\ $y$ is uniquely determined by $x$ and independent of the specific choice of the sequence $(u_n)$ with the above properties). Moreover, $-A$ generates a quasi-contractive holomorphic $C_0$-semigroup $(S(t))_{t\ge 0}$ on $H$.
We call $A$ the \emphdef{operator associated with} $(\fa,j)$ and write $A\sim (\fa,j)$.
Similarly, $(S(t))_{t\ge 0}$ is called the \emphdef{semigroup associated with} $(\fa,j)$ and we write $(S(t))_{t\ge 0}\sim (\fa,j)$.
We refer to~\cite[Section~3 and specifically Theorem~3.2]{AtE12:sect-form} for these and related results.

Let $(\fa,j)$ be a sectorial form in $H$. Suppose that $A\sim (\fa,j)$  and $(S(t))_{t\ge 0}\sim (\fa,j)$.
The following criterion for the invariance of a closed, convex subset $C$ of $H$ under the associated semigroup is very convenient. Denote by $P\colon H\to C$ the orthogonal projection onto $C$.

\begin{prop}\label{prop:suff-inv}
Assume that for all $u\in D(\fa)$ there exists a $w\in D(\fa)$ such that $Pj(u)=j(w)$ and $\Re\fa(w,u-w)\ge 0$.
Then $S(t)C\subset C$ for all $t\ge 0$,
and there exists a $\lambda_0\ge 0$ such that $\lambda R(\lambda,-A)C\subset C$ for all $\lambda>\lambda_0$.
\end{prop}
\begin{rem}
Proposition~\ref{prop:suff-inv} is an immediate consequence of~\cite[Proposition~3.11]{AtE12:sect-form}. We point out that the accretivity hypothesis made there and in~\cite[Proposition~2.9]{AtE12:sect-form} can be omitted. 
This was already observed in~\cite[Theorem~2.1]{MVV05}. In fact, due to Euler's exponential formula for the associated $C_0$-semigroup it suffices to show that $\lambda R(\lambda,-A)C\subset C$ for $\lambda>0$ sufficiently large in order to ensure invariance.
\end{rem}

We consider a special case. Assume that $H=L^2(X)$, where $(X,\Sigma,\mu)$ is a finite measure space.
Then $C_\infty = \{ f\in H : \abs{f(x)}\le 1\text{ a.e.}\}$ is a closed, convex subset of $H$ 
and the orthogonal projection $Q$ of $H$ onto $C_\infty$ is given by
\begin{equation}
    Qf = (\abs{f}\wedge \one_X)\sign f,
\end{equation}
where $\sign f = \one_{[f\ne 0]}\frac{f}{\abs{f}}$.
The set $C_\infty$ is invariant under $(S(t))_{t\ge 0}$ if and only if the semigroup is \emphdef{$L^\infty$-contractive}, i.e. 
\[
    \norm{S(t)f}_{L^\infty(X)}\le\norm{f}_{L^\infty(X)}
\]
for all $f\in L^\infty(X)$ and $t>0$.
In this case, we consider the \emphdef{part} of $A$ in $L^\infty(X)$, i.e.\ in the operator
\begin{align*}
    D(A_\infty) &= \{ u\in D(A)\cap L^\infty(X) : Au\in L^\infty(X)\},\\
    A_\infty u &= Au.
\end{align*}

By a result of Lotz~\cite[Corollary~4.3.19]{ABHN2011}, the operator $-A_\infty$ will generate a $C_0$-semigroup on $L^\infty(X)$ only if it is bounded. Still, we obtain the following as a consequence of Proposition~\ref{prop:suff-inv}.

\begin{cor}\label{cor:suff-inv}
Assume that for all $u\in D(\fa)$ there exists a $w\in D(\fa)$ such that $Qj(u)=j(w)$ and $\Re\fa(w,u-w)\ge 0$.
Then $-A_\infty$ is m-dissipative in $L^\infty(X)$, i.e.\ $(0,\infty)\in\rho(-A_\infty)$ and $\norm{\lambda R(\lambda,-A_\infty)}_{\Linop(L^\infty(X))}\le 1$ for all $\lambda>0$.
\end{cor}
\begin{proof}
By Proposition~\ref{prop:suff-inv} there exists $\lambda_0\ge 0$ such that $\lambda R(\lambda,-A)C_\infty\subset C_\infty$ for all $\lambda>\lambda_0$.
Thus $(\lambda_0,\infty)\subset\rho(-A_\infty)$ and $R(\lambda,-A_\infty) = \restrict{R(\lambda,-A)}{L^\infty(X)}$.
Now it follows from the proof of~\cite[Lemma~3.4.2]{ABHN2011} that $-A_\infty$ is m-dissipative.
\end{proof}

\section{The Wentzell Laplacian in \texorpdfstring{$L^2(\Omega)\oplus L^2(\Gamma)$}{L²(Ω)⊕L²(Γ)}}\label{sec:W-Lap-apptrace}

Let $\Omega\subset\RR^d$ be open and bounded, $\Gamma = \partial\Omega$. 
We assume that $\sigma$ is a fixed finite positive Borel measure on $\Gamma$
and we also simply write $L^2(\Gamma)$ instead of $L^2(\Gamma,\sigma)$.
We briefly comment on this specific setting.
\begin{rem}
\begin{parenum}
\item 
If $\Hm(\Gamma)<\infty$, the choice $\sigma=\Hm\niv\Gamma$ is very natural from a geometric point of view.
This is the case for a Lipschitz domain, but also allows to consider domains with inward and outward pointing cusps, for example.

\item
The setting here, however, also works for domains with fractional boundary parts, as long as a suitable finite Borel measure $\sigma$ on $\Gamma$ is available. For example, if $\Omega$ is in addition connected, then $\sigma$ can be chosen to be the harmonic measure with respect to any fixed base point in $\Omega$.

More specifically, if $\Omega\subset\RR^2$ is the Koch snowflake domain, the choice $\sigma = \Hm[s]\niv \Gamma$ for the Hausdorff dimension $s=\frac{\log 4}{\log 3}$ of the Koch curve is very natural, since then there exist constants $c_1,c_2>0$ such that
\[
    c_1 r^s \le \Hm[s](\Gamma\cap B(z,r)) \le c_2 r^s.
\]
for all $x\in \Gamma$ and $r\in(0,1]$ (i.e.~$\Gamma$ is an \emphdef{$s$-set}), see~\cite[Proposition~4.2]{Lan02}. Therefore $\Hm[s]\niv\Gamma$ is a finite Borel measure on $\Gamma$. The Koch snowflake domain is actually quite well-behaved since it has the Sobolev extension property and there is a bounded trace operator on $H^1(\Omega)$ with values in the Besov space $B^{2,2}_{s/2}(\Gamma)$, see~\cite[Propositions~2 and~3]{Wal91}.

\item 
We make the assumptions that $\Omega$ is bounded and $\sigma$ is finite only for simplicity. Using the same techniques as in~\cite{AW03} (enforcing Dirichlet conditions on the part of $\Gamma$ where $\sigma$ is not locally finite), throughout this section we could allow arbitrary open sets $\Omega$ and arbitrary positive Borel measures $\sigma$ on $\Gamma$. To this end we would need to suitably replace $C(\clos{\Omega})\cap H^1(\Omega)$ by 
\[
    \Bigl\{ u\in C(\clos{\Omega})\cap H^1(\Omega): \int_\Gamma \abs{u}^2\dx[\sigma]<\infty\Bigr\}.
\] 
The assumptions that we make here ensure that for $u\in C(\clos{\Omega})\cap H^1(\Omega)$ one has $\restrict{u}{\Gamma}\in L^2(\Gamma,\sigma)$ and that the set of these restrictions is dense in $L^2(\Gamma,\sigma)$, see~\cite[Theorem~7.8 and Proposition~7.9]{Fo99}.
\end{parenum}
\end{rem}

\begin{samepage}
\begin{definition}
Let $u\in H^1(\Omega)$.
\begin{alenum}
\item Let $\phi\in L^2(\Gamma,\sigma)$. We say that $\phi$ is an \emphdef{(approximative) trace} (with respect to $\sigma$) of $u$ if there exists a sequence $(u_n)_{n\in\NN}$ in $C(\clos{\Omega})\cap H^1(\Omega)$ such that $u_n\to u$ in $H^1(\Omega)$ and $\restrict{u_n}{\Gamma}\to \phi$ in $L^2(\Gamma,\sigma)$.

\item We set $\Tr(u) := \{ \phi\in L^2(\Gamma,\sigma): \text{$\phi$ is a trace of $u$}\}$.
\end{alenum}
\end{definition}
\end{samepage}
\begin{rem*}
Suppose that $\Hm(\Gamma)<\infty$ and $\sigma=\Hm\niv\Gamma$.
Let $u\in H^1(\Omega)$. Then it may happen that $\Tr(u)=\emptyset$, and even if $u$ has a trace it may not be unique. In fact, it is possible that $\Tr(0)$ is infinite, cf.~\cite[Examples~4.2 and~4.3]{AW03} or~\cite[Section~4]{AtE12:sect-form}. However, if $\Omega$ has continuous boundary in the sense of graphs or if $d=2$ and $\Omega$ is connected, then $\Tr(u)$ has at most one element, see~\cite[Theorem~4.11 and Corollary~5.4]{Sau2020}. 
\end{rem*}

Next, we define the normal derivative via Green's formula.
\begin{definition}
Let $u\in H^1(\Omega)$ be such that $\Delta u\in L^2(\Omega)$.
Let $h\in L^2(\Gamma,\sigma)$. We set $\partial_\nor u=h$ if
\begin{equation}
    \int_\Omega (\Delta u)\conj{v} + \int_\Omega \nabla u\cdot\conj{\nabla v} = \int_\Gamma h\conj{v}\dx[\sigma]
\end{equation}
for all $v\in C(\clos{\Omega})\cap H^1(\Omega)$ and call $\partial_\nor u$ the \emphdef{normal derivative} (with respect to $\sigma$) of $u$.
If $h$ can be chosen in $C(\Gamma)$, we write $\partial_\nor u\in C(\Gamma)$.
\end{definition}
\begin{rem*}
\begin{parenum}
\item 
If a normal derivative exists, it is uniquely determined in $L^2(\Gamma,\sigma)$ due to the density of $\{\restrict{v}{\Gamma} : v\in C(\clos{\Omega})\cap H^1(\Omega)\}$ in $L^2(\Gamma,\sigma)$. Of course, if $\sigma$ vanishes locally on parts of $\Gamma$, then the normal derivative might not be uniquely determined in $C(\Gamma)$.
\item
If $\Omega$ is Lipschitz and $\sigma=\Hm\niv\Gamma$, then this definition is motivated by the divergence theorem.
However, if $\Omega\subset\RR^2$ is the Koch snowflake domain and $\sigma=\Hm[s]\niv\Gamma$ for $s=\frac{\log 4}{\log 3}$, the above normal derivative represents the normal derivative given as a bounded linear functional on the Besov space $B^{2,2}_{s/2}(\Gamma)$ as in~\cite[Theorem~4.15]{Lan02} via integration in $L^2(\Gamma,\sigma)$.
\end{parenum}
\end{rem*}

One cannot realise the Laplacian with Wentzell boundary conditions as an $m$-sectorial operator in $L^2(\Omega)$.
If $\Delta u$ is merely in $L^2(\Omega)$, one cannot sensibly consider its boundary trace.
As a natural relaxation, however, one can decouple the boundary conditions. 
This can be accomplished by defining a realisation $\Delta^\rmW$ of the Wentzell Laplacian in $L^2(\Omega)\oplus L^2(\Gamma)$ in the following way.

\begin{thm}\label{thm:gen-Wentzell-H}
Suppose that $\Omega\subset\RR^d$ is open and bounded, and that its boundary $\Gamma$ is equipped with a finite positive Borel measure $\sigma$.
Let $\beta\in L^\infty(\Gamma)$ be a complex-valued function.
Define the operator $\Delta^\rmW$ in $L^2(\Omega)\oplus L^2(\Gamma)$ by
\begin{align*}
    D(\Delta^\rmW) &= \{ (u,\phi)\in u\in H^1(\Omega)\oplus L^2(\Gamma) : \Delta u\in L^2(\Omega),\ \phi\in\Tr(u),\ \partial_\nor u\in L^2(\Gamma)\},\\
    \Delta^\rmW (u, \phi) &= (\Delta u, -\partial_\nor u - \beta \phi).
\end{align*}
Then $\Delta^\rmW$ generates a holomorphic $C_0$-semigroup $(S(t))_{t\ge 0}$ on $L^2(\Omega)\oplus L^2(\Gamma)$.
\end{thm}
\begin{rem*}
We call $\Delta^\rmW$ the \emphdef{Wentzell Laplacian} in $L^2(\Omega)\oplus L^2(\Gamma)$ (with respect to $\sigma$ and $\beta$), even though in general it does not even make sense to ask whether $-\partial_\nor u - \beta\phi$ is a trace of $\Delta u$.
In Section~\ref{sec:Wentzell-C} we will obtain a realisation of the Wentzell Laplacian in $C(\clos{\Omega})$,
and in Section~\ref{sec:Wentzell-H} a realisation in a suitable space of $H^1(\Omega)$ functions with traces; both are in a suitable sense parts of the operator $\Delta^\rmW$ for which Wentzell boundary conditions are realised in the sense of traces.
\end{rem*}
\begin{proof}
Let $D(\fa)=C(\clos{\Omega})\cap H^1(\Omega)$ and define 
\[
    \fa(u,v) = \int_\Omega\nabla u\cdot\conj{\nabla v} + \int_\Gamma \beta u\conj{v}\dx[\sigma]
\]
for all $u,v\in D(\fa)$. We define $j\colon D(\fa)\to H$ by $j(u)=(u,\restrict{u}{\Gamma})$.

\begin{parenum}
\item We show that $j(D(\fa))$ is dense in $H=L^2(\Omega)\oplus L^2(\Gamma)$.

\vspace*{4pt}\noindent First step: Let $\Phi\in C^\infty(\RR^d)$. We show that $(0,\restrict{\Phi}{\Gamma})\in \clos{j(D(\fa))}$. Let $K_n\subset\Omega$ be compact for all $n\in\NN$ such that $\bigcup_{n\in\NN}K_n=\Omega$. For all $n\in\NN$ there exists $\chi_n\in C^\infty_\cpt(\RR^d)$ such that $0\le\chi_n\le 1$, $\chi_n\ident 1$ on $\Gamma$ and $\chi_n\ident 0$ on $K_n$. Thus $\chi_n\Phi=\Phi$ on $\Gamma$ and $\chi_n\Phi\to 0$ in $L^2(\Omega)$. This proves the claim.

\vspace*{4pt}\noindent Second step: It follows from the first step that $\{0\}\times L^2(\Gamma)\subset\clos{j(D(\fa))}$.

\vspace*{4pt}\noindent Third step: Let $(f,g)\in H$. Then $(0,g)\in\clos{j(D(\fa))}$ by the second step.
There exists a sequence $(u_n)_{n\in\NN}$ in $C^\infty_\cpt(\Omega)$ such that $u_n\to f$ in $L^2(\Omega)$. Thus $(f,0)\in\clos{j(D(\fa))}$. We conclude that $(f,g)=(f,0) + (0,g)\in\clos{j(D(\fa))}$.

\item It is easy to see that there exist $\omega\ge 0$ and $\theta\in[0,\frac{\pi}{2})$ such that $\fa(u,u)+\omega\norm{j(u)}_H^2\in\Sigma_\theta$ for all $u\in D(\fa)$.
Let $A\sim (\fa,j)$. We know from Section~\ref{sec:fm-incompl} that $-A$ generates a holomorphic $C_0$-semigroup $(S(t))_{t\ge 0}$ on $H$. We show that $A=-\Delta^\rmW$.

\item We show $A\subset -\Delta^\rmW$. Let $(u,\phi)\in D(A)$ with $A(u,\phi)=(f,g)\in H$.
By the definition of $A$ there exists $(u_n)_{n\in\NN}$ in $D(\fa)=C(\clos{\Omega})\cap H^1(\Omega)$ such that $j(u_n)=(u_n,\restrict{u_n}{\Gamma})\to (u,\phi)$ in $H$,
$\fa(u_n-u_m,u_n-u_m)\to 0$ as $n,m\to\infty$ and
\begin{equation}\label{eq:3.2-form-op-id}
    \lim_{n\to\infty}\fa(u_n,v) = \int_\Omega f\conj{v} + \int_\Gamma g\conj{v}\dx[\sigma]
\end{equation}
for all $v\in D(\fa)$. Thus $u_n\to u$ in $L^2(\Omega)$ and $\restrict{u_n}{\Gamma}\to\phi$ in $L^2(\Gamma)$. Moreover,
\[
    \int_\Omega\abs{\nabla (u_n-u_m)}^2 + \int_\Gamma\beta\abs{u_n-u_m}^2\dx[\sigma]\to 0
\]
as $n,m\to\infty$. Since $(\restrict{u_n}{\Gamma})_{n\in\NN}$ converges in $L^2(\Gamma)$, it follows that
\[
    \int_\Omega\abs{\nabla (u_n-u_m)}^2\to 0
\]
as $n,m\to\infty$. Thus $(u_n)_{n\in\NN}$ is a Cauchy sequence in $H^1(\Omega)$.
It follows that $u\in H^1(\Omega)$ and $\lim_{n\to\infty} u_n=u$ in $H^1(\Omega)$.
Consequently, $\phi\in\Tr(u)$. Letting $n\to\infty$ in~\eqref{eq:3.2-form-op-id} shows that
\begin{equation}\label{eq:3.3}
    \int_\Omega\nabla u\cdot\conj{\nabla v} + \int_\Gamma\beta \phi\conj{v}\dx[\sigma] = \int_\Omega f\conj{v} + \int_\Gamma g\conj{v}\dx[\sigma]
\end{equation}
for all $v\in C(\clos{\Omega})\cap H^1(\Omega)$. Taking $v\in C^\infty_\cpt(\Omega)$ we deduce that $-\Delta u = f$. Now~\eqref{eq:3.3} gives
\begin{equation}
    \int_\Omega(\Delta u)\conj{v} + \int_\Omega\nabla u\cdot\conj{\nabla v} = \int_\Gamma(g-\beta\phi)\conj{v}\dx[\sigma]
\end{equation}
for all $v\in C(\clos{\Omega})\cap H^1(\Omega)$. Thus $\partial_\nor u = g-\beta\phi$.
We have shown that $u\in H^1(\Omega)$, $-\Delta u = f$, $\phi\in\Tr(u)$ and $\partial_\nor u+\beta\phi = g$. Thus $(u,\phi)\in D(\Delta^\rmW)$ and $A(u,\phi) = (-\Delta u,\partial_\nor u + \beta\phi) = -\Delta^\rmW(u,\phi)$.

\item We show $-\Delta^\rmW\subset A$. Let $(u,\phi)\in D(\Delta^\rmW)$ and $-\Delta^\rmW(u,\phi)=(f,g)\in H$. 
Then $u\in H^1(\Omega)$, $-\Delta u = f$, $\phi\in\Tr(u)$ and $\partial_\nor u+\beta\phi = g$.
Since $\phi\in\Tr(u)$, there exists $(u_n)_{n\in\NN}$ in $C(\clos{\Omega})\cap H^1(\Omega)$ such that $u_n\to u$ in $H^1(\Omega)$ and $\restrict{u_n}{\Gamma}\to\phi$ in $L^2(\Omega)$. 
This implies that $j(u_n)\to (u,\phi)$ in $H$ and $\fa(u_n-u_m,u_n-u_m)\to 0$ as $n,m\to\infty$.
Moreover,
\[
    \fa(u_n,v) = \int_\Omega\nabla u_n\cdot\conj{\nabla v} + \int_\Gamma\beta u_n\conj{v}\dx[\sigma]
        \to \int_\Omega \nabla u\cdot\conj{\nabla v} + \int_\Gamma\beta \phi\conj{v}\dx[\sigma]
\]
for all $v\in D(\fa)$. Since $\int_\Omega (\Delta u)\conj{v} + \int_\Omega f\conj{v}=0$, it follows that
\begin{align*}
    \int_\Omega\nabla u\cdot\conj{\nabla v} + \int_\Gamma\beta\phi\conj{v}\dx[\sigma] &= \int_\Omega (\Delta u)\conj{v} + \int_\Omega \nabla u\cdot\conj{\nabla v} + \int_\Omega f\conj{v} + \int_\Gamma\beta\phi\conj{v}\dx[\sigma] \\
        &= \int_\Gamma(\partial_\nor u)\conj{v}\dx[\sigma] + \int_\Omega f\conj{v} + \int_\Gamma\beta \phi\conj{v}\dx[\sigma] 
    = \int_\Omega f\conj{v} + \int_\Gamma g\conj{v}\dx[\sigma]
\end{align*}
for all $v\in D(\fa)$.
Thus, by the definition of $A$, $(u,\phi)\in D(A)$ and $A(u,\phi)=(f,g)$.\qedhere
\end{parenum}
\end{proof}

Consider the orthogonal projection $Q$ from $L^2(\Omega)$ onto the set $C_\infty$ as in Section~\ref{sec:fm-incompl} for the choice $X=(\Omega,\mathcal{B}(\Omega),\abs{{}\cdot{}})$.
\begin{lem}\label{lem:ouh-Q-ineq}
If $u\in H^1(\Omega)$, then $Qu\in H^1(\Omega)$ and
\begin{equation}\label{eq:Q-ineq}
    \Re\int_\Omega\nabla Qu\cdot\conj{\nabla(u-Qu)}\ge 0.
\end{equation}
Moreover, $Qu\in C(\clos{\Omega})$ if $u\in C(\clos{\Omega})\cap H^1(\Omega)$.
\end{lem}
For the proof we refer to~\cite[Proposition~4.11]{Ouhabaz2005} or to the Appendix, where we give a less technical and short abstract argument.

Next, we show that the semigroup generated by $\Delta^\rmW -\omega_0$ is $L^\infty$-contractive if $\omega_0\ge0$ is sufficiently large.
To this end, we consider the disjoint union of measure spaces $(X,\mu)$ where $X=\Omega\sqcup\Gamma=\clos{\Omega}$ and $\mu$ is the sum of the $d$-dimensional Lebesgue measure on $\Omega$ and $\sigma$ on $\Gamma$. We then identify $L^2(\Omega)\oplus L^2(\Gamma)$ with $L^2(X,\mu)$, as well as $L^\infty(\Omega)\oplus L^\infty(\Gamma)$ with $L^\infty(X,\mu)$.
We prove that Proposition~\ref{prop:suff-inv} can be applied to a shifted version of the form $\fa$.

\begin{prop}\label{prop:W-Lap-contr}
Adopt the notation and assumptions of Theorem~\ref{thm:gen-Wentzell-H}.
There exists an $\omega_0\ge0$ such that the $C_0$-semigroup generated by $\Delta^\rmW -\omega_0$ is $L^\infty$-contractive.
\end{prop}
\begin{proof}
We have $-\Delta^\rmW\sim (\fa,j)$, where $D(\fa)=C(\clos{\Omega})\cap H^1(\Omega)$, 
\[
    \fa(u,v)=\int_\Omega\nabla u\cdot\conj{\nabla v} + \int_\Gamma\beta u\conj{v}\dx[\sigma],
\]
and $j\colon D(\fa)\to L^2(\Omega)\oplus L^2(\Gamma)$ is given by $j(u)=(u,\restrict{u}{\Gamma})$; see the proof of Theorem~\ref{thm:gen-Wentzell-H}.
Thus $-\Delta^\rmW+\omega_0$ is associated with $(\fa_{\omega_0},j)$, where $D(\fa_{\omega_0})=D(\fa)$ and
\begin{equation}\label{eq:3.6}
    \fa_{\omega_0}(u,v) = \fa(u,v)+\omega_0\Bigl(\int_\Omega u\conj{v} + \int_\Gamma u\conj{v}\dx[\sigma]\Bigr).
\end{equation}
Let 
\[
    \widetilde{C}_\infty := \{ (u,\phi)\in L^\infty(\Omega)\oplus L^\infty(\Gamma) : \norm{u}_{L^\infty(\Omega)}\le 1,\ \norm{\phi}_{L^\infty(\Gamma)}\le 1\}.
\]
Then the orthogonal projection $\widetilde{Q}$ of $L^2(\Omega)\oplus L^2(\Gamma)$ onto $\widetilde{C}_\infty$ is given by
\[
    \widetilde{Q}(u,\phi) = (Qu, Q_\Gamma\phi),
\]
where $Q_\Gamma\phi =  (\abs{\phi}\wedge \one_\Gamma)\sign\phi$.
Let $u\in D(\fa)$. The above shows that $\widetilde{Q} j(u) = j(Qu)$
since $Qu\in D(\fa)$ by Lemma~\ref{lem:ouh-Q-ineq}. 
So it suffices to show that $\Re\fa_{\omega_0}(Qu,u-Qu)\ge 0$. Then $L^\infty$-contractivity follows from Proposition~\ref{prop:suff-inv}.

Note that 
\[
    u-Qu = (\abs{u}-\abs{u}\wedge\one_\Omega)\sign u = (\abs{u}-\one_\Omega)^+\sign u.
\]
Thus $Qu\conj{(u-Qu)} = (\abs{u}-\one_\Omega)^+$. Similarly
\[
   Q_\Gamma \phi \conj{(\phi-Q_\Gamma\phi)} = (\abs{\phi}-\one_\Gamma)^+. 
\]
Therefore by Lemma~\ref{lem:ouh-Q-ineq} one has
\begin{align*}
    \Re\fa_{\omega_0}(Qu,u-Qu) &\ge \Re\int_\Omega\omega_0Qu\conj{(u-Qu)} + \Re\int_\Gamma(\omega_0+\beta)Q_\Gamma\phi\conj{(\phi-Q_\Gamma\phi)}\dx[\sigma] \\
    &= \int_\Omega \omega_0(\abs{u}-\one_\Omega)^+ + \int_\Gamma(\omega_0+\Re\beta)(\abs{\phi}-\one_\Gamma)^+\dx[\sigma]\ge 0
\end{align*}
if $\omega_0\ge0$ is so large that $\omega_0+\Re\beta\ge 0$.
\end{proof}

\section{The Wentzell Laplacian in \texorpdfstring{$C(\clos{\Omega})$}{C(Ω̅)}}\label{sec:Wentzell-C}

Throughout this section we suppose that $\Omega\subset\RR^d$ is a bounded open set with Lipschitz boundary. Moreover, we equip $\Gamma=\partial\Omega$ with the Borel measure $\Hm\niv\Gamma$. Note that $\Hm(\Gamma)<\infty$.

Then for each $u\in H^1(\Omega)$ there exists a unique $u_\Gamma\in L^2(\Gamma)$ such that $\Tr(u) = \{u_\Gamma\}$.
Moreover, the mapping $u\mapsto u_\Gamma\colon H^1(\Omega)\to L^2(\Gamma)$ is continuous.
The space $C(\clos{\Omega})\cap H^1(\Omega)$ is dense in $H^1(\Omega)$ and $u_\Gamma=\restrict{u}{\Gamma}$ for all $u\in C(\clos{\Omega})\cap H^1(\Omega)$. As before $\beta\in L^\infty(\Gamma)$ is allowed to be complex-valued.

Our aim is to prove the following.
\begin{thm}\label{thm:Wentzell-C}
Suppose that $\Omega\subset\RR^d$ is a Lipschitz domain and that its boundary $\Gamma$ is equipped with the measure $\Hm\niv\Gamma$. Let $\beta\in L^\infty(\Gamma)$ be a complex-valued function.
Define the operator $\Delta^\rmW_\rmC$ on $C(\clos{\Omega})$ by
\begin{align*}
    D(\Delta^\rmW_\rmC) &= \{ u\in C(\clos{\Omega})\cap H^1(\Omega) : \Delta u\in C(\clos{\Omega}),\ \partial_\nor u +\beta \restrict{u}{\Gamma} + \restrict{(\Delta u)}{\Gamma} = 0\},\\
    \Delta^\rmW_\rmC u &= \Delta u.
\end{align*}
Then $\Delta^\rmW_\rmC$ generates a $C_0$-semigroup on $C(\clos{\Omega})$.
\end{thm}

We need in an essential way both an $L^\infty$-estimate and continuity up to the boundary for solutions $u$ of the elliptic equation
\begin{equation}\label{eq:robin-reg}
    \left\{\begin{aligned}
    \lambda u - \Delta u &= f, \\
    \partial_\nor u + (\lambda+\beta)u_\Gamma &= g,
    \end{aligned}\right.
\end{equation}
where $f\in L^p(\Omega)$ and $g\in L^q(\Gamma)$ for appropriate $p,q$ are given and $\lambda>0$ is sufficiently large.

For the following we choose $\omega_0\ge 0$ so large that Proposition~\ref{prop:W-Lap-contr} holds and
the form $\fa_{\omega_0}\colon H^1(\Omega)\times H^1(\Omega)\to\CC$ given by
\[
    \fa_{\omega_0}(u,v) = \int_{\Omega}\nabla u\cdot\conj{\nabla v} + \int_\Gamma\beta u_\Gamma\conj{v}_\Gamma + \omega_0\Bigl(\int_\Omega u\conj{v} + \int_\Gamma u_\Gamma\conj{v}_\Gamma\Bigr)
\]
is coercive. So there exists an $\alpha>0$ such that
\begin{equation}
    \Re \fa_{\omega_0}(v,v)\ge \alpha\norm{v}_{H^1(\Omega)}^2 
\end{equation}
for all $v\in H^1(\Omega)$.
Suppose that $f\in L^p(\Omega)$ and $g\in L^q(\Gamma)$, where $p>\frac{d}{2}$ and $q>d-1$.
Then $F\colon v\mapsto\int_\Omega f\conj{v} + \int_\Gamma g\conj{v}_\Gamma$ is a continuous anti-linear functional on $H^1(\Omega)$. In fact, by the Sobolev embedding theorem $H^1(\Omega)\hookrightarrow L^{2^*}(\Omega)$, where $2^*=\frac{2d}{2-d}$ if $d>2$ and for all $2^*\in[1,\infty)$ if $d=2$. 
Moreover, the trace maps $H^1(\Omega)$ continuously into $L^s(\Gamma)$, where $s=\frac{2(d-1)}{d-2}$ if $d>2$ and for all $s\in[1,\infty)$ if $d=2$; see~\cite[Theorems~4.2 and~4.6]{Nec12}. 
Hence, after $s$ is fixed suitably, there is a constant $c_1>0$ such that
\begin{equation}\label{eq:trace-Ls-est}
    \norm{u_\Gamma}_{L^s(\Gamma)} \le c_1\norm{u}_{H^1(\Omega)}
\end{equation}
for all $u\in H^1(\Omega)$. 
It follows from Hölder's inequality that $F$ is continuous.
So there exists a unique solution $u\in H^1(\Omega)$ of~\eqref{eq:robin-reg} for $\lambda=\omega_0$, i.e.\ $u$ satisfies
\[
    \fa_{\omega_0}(u,v) = F(v) = \int_\Omega f\conj{v} + \int_\Gamma g\conj{v}_\Gamma
\]
for all $v\in H^1(\Omega)$.

We start with the $L^\infty$-estimate in Proposition~\ref{prop:robin-bdd}.
For real $\beta\ge0$ it was proved by Warma in~\cite[Theorem~2.2]{War06}. The argument remains valid for complex $\beta$.
We reproduce Warma's proof with slight adjustments in order to make the exposition more self-contained and since the proof uses a neat argument to simultaneously control the behaviour in the inside and on the boundary.
The following stopping lemma by Stampacchia is an analytical tool in the proof. For the reader's convenience, we provide the proof.
\begin{lem}[{\cite[Lemma~4.1\,(i)]{Sta64}}]\label{lem:Sta-stop-lem}
Let $\phi\colon[0,\infty)\to[0,\infty)$ be monotonically non-increasing such that for some $\alpha, c_\phi>0$ and $\delta>1$ one has
\begin{equation}\label{eq:Sta-stop-ineq}
    \phi(h)\le c_\phi(h-k)^{-\alpha}\phi(k)^\delta
\end{equation}
for all $0\le k < h$.
Then $\phi(h)=0$ for all $h\ge t_0 := (c_\phi)^{1/\alpha}\phi(0)^{(\delta-1)/\alpha} 2^{\delta/(\delta-1)}$.
\end{lem}
\begin{proof}
Let $\gamma=-\frac{1}{\delta-1}$. Then $1+\gamma\delta=\gamma$.
After replacing $\phi$ by $c_\phi^{-\gamma}\phi$, we may suppose that $c_\phi=1$.
Let $k_m = (1-2^{-m})t_0$ for all $m\in\NN_0$. Then by~\eqref{eq:Sta-stop-ineq} one has
\begin{equation}\label{eq:phi-km-est1}
    \phi(k_m)\le (k_m-k_{m-1})^{-\alpha}\phi(k_{m-1})^\delta = \frac{2^{m\alpha}}{t_0^\alpha} \phi(k_{m-1})^\delta
\end{equation}
for all $m\in\NN$. We prove by induction on $m$ that
\begin{equation}\label{eq:phi-km-est2}
    \phi(k_m) \le 2^{m\alpha\gamma}\phi(0)
\end{equation}
for all $m\in\NN_0$. This is trivial for $m=0$. If $m\in\NN$ is such that~\eqref{eq:phi-km-est2} holds for $m-1$, using~\eqref{eq:phi-km-est1} we obtain
\begin{align*}
    \phi(k_m) &\le \frac{2^{m\alpha}}{t_0^\alpha} \Bigl(2^{(m-1)\alpha\gamma}\phi(0)\Bigr)^\delta \\
        & = 2^{m\alpha+(m-1)\alpha\gamma\delta+\alpha\gamma\delta}\phi(0)\frac{\phi(0)^{\delta-1}2^{-\alpha\gamma\delta}}{t_0^\alpha} = 2^{m\alpha\gamma}\phi(0).
\end{align*}

As $\phi\ge0$ and $\alpha\gamma<0$, it follows from~\eqref{eq:phi-km-est2} that $\lim_{m\to\infty}\phi(k_m)=0$.
This implies $\phi(t_0)=0$ since $(k_m)$ monotonically increases to $t_0$ and $\phi$ is monotonically non-increasing.
\end{proof}

\begin{prop}\label{prop:robin-bdd}
Let $p>\frac{d}{2}$ and $q>d-1$. Let $f\in L^p(\Omega)$, $g\in L^q(\Gamma)$ and $\lambda\ge \omega_0$. Then the solution $u\in H^1(\Omega)$ of~\eqref{eq:robin-reg} is in $L^\infty(\Omega)$. 
Moreover, there exists a constant $c>0$ independent of $f$ and $g$ such that
\begin{equation}\label{eq:robin-bdd-est}
    \norm{u}_{L^\infty(\Omega)}\le c(\norm{f}_{L^p(\Omega)} + \norm{g}_{L^q(\Gamma)}).
\end{equation}
\end{prop}
\begin{proof}
We may suppose that $\lambda=\omega_0$.
We use the relational symbol `$\leconst$' to express a less-than-or-equal inequality up to an negligible multiplicative positive constant, which may change from line to line.
We let $2^*=\frac{2d}{d-2}$ and $s=\frac{2(d-1)}{d-2}$ if $d>2$, and fix any $2^*>\frac{2p}{p-1}$ and $s>\frac{2q}{q-1}$ for $d=2$.
For $k\ge 0$ we let $Q_ku = (\abs{u}\wedge k)\sign u$ and $v_k=u-Q_k u=(\abs{u}-k)^+\sign u$.
Then by~\eqref{eq:Q-ineq} one has
\begin{equation}\label{eq:4.6}
    \Re \int_{\Omega}\nabla u\conj{\nabla v_k} = \Re \int_\Omega \nabla(u-Q_ku)\conj{\nabla v_k} + \Re\int_\Omega\nabla (Q_k u)\conj{\nabla v_k} \ge \int_\Omega \abs{\nabla v_k}^2.
\end{equation}
Moreover, note that $u\conj{v_k} = \abs{u}(\abs{u}-k)^+\ge\bigl((\abs{u}-k)^+\bigr)^2 = \abs{v_k}^2$.
Hence by~\eqref{eq:4.6} 
\begin{align*}
    \Re\Bigl(\int_\Omega f\conj{v_k} + \int_\Gamma g\conj{v_k}\Bigr) = \fa_{\omega_0}(u,v_k) 
        &\ge \omega_0\int_\Omega \abs{v_k}^2 + \int_\Omega \abs{\nabla v_k}^2 + \Re\int_\Gamma (\beta+\omega_0)\abs{(v_k)_\Gamma}^2 \\
        &=\Re \fa_{\omega_0}(v_k,v_k)\ge\alpha \norm{v_k}_{H^1(\Omega)}^2.
\end{align*}
Let $\Omega_k = \{x\in\Omega: \abs{u(x)}>k\}$ and $\Gamma_k = \{ z\in\Gamma : \abs{u_\Gamma(z)}>k\}$.
Let $\sigma = \Hm\niv\Gamma$.
Then by Hölder's inequality
\begin{align*}
    \alpha\norm{v_k}_{H^1(\Omega)}^2 &\le \norm{f\one_{\Omega_k}}_{L^2(\Omega)}\norm{v_k\one_{\Omega_k}}_{L^2(\Omega)} + \norm{g\one_{\Gamma_k}}_{L^2(\Gamma)}\norm{(v_k)_\Gamma\one_{\Gamma_k}}_{L^2(\Gamma)} \\
        & \le \norm{f}_{L^p(\Omega)}\abs{\Omega_k}^{\frac{1}{2}-\frac{1}{p}}\norm{v_k}_{L^{2^*}(\Omega)}\abs{\Omega_k}^{\frac{1}{2}-\frac{1}{2^*}} + \norm{g}_{L^q(\Gamma)}\sigma(\Gamma_k)^{\frac{1}{2}-\frac{1}{q}}\norm{(v_k)_\Gamma}_{L^s(\Gamma)}\sigma(\Gamma_k)^{\frac{1}{2}-\frac{1}{s}} \\
        & \leconst \Bigl(\abs{\Omega_k}^{1-\frac{1}{p}-\frac{1}{2^*}} + \sigma(\Gamma_k)^{1-\frac{1}{q}-\frac{1}{s}}\Bigr)\norm{v_k}_{H^1(\Omega)},
\end{align*}
where we used the Sobolev embedding $H^1(\Omega)\hookrightarrow L^{2^*}(\Omega)$ for the estimate on $\Omega$ and~\eqref{eq:trace-Ls-est} for the estimate on $\Gamma$.
So with
\[
    H(k) = \abs{\Omega_k}^{1-\frac{1}{p}-\frac{1}{2^*}} + \sigma(\Gamma_k)^{1-\frac{1}{q}-\frac{1}{s}}
\]
one has $\norm{v_k}_{H^1(\Omega)}\leconst H(k)$. Again by the Sobolev embedding theorem
it follows that
\begin{equation}\label{eq:Hk-inner-est}
    \norm{v_k}_{L^{2^*}(\Omega)} \leconst H(k),
\end{equation}
and by~\eqref{eq:trace-Ls-est} we obtain
\begin{equation}\label{eq:Hk-bdy-est}
    \norm{(v_k)_\Gamma}_{L^s(\Gamma)}\leconst H(k).
\end{equation}

Now fix $k\ge 0$ and let $h>k$. Then one has $\Omega_h\subset\Omega_k$ and $\abs{v_k} = (\abs{u}-k)^+\ge h-k$ on $\Omega_h$. Correspondingly $\Gamma_h\subset\Gamma_k$ and $\abs{(v_k)_\Gamma}\ge h-k$ on $\Gamma_h$.
Thus~\eqref{eq:Hk-inner-est} and~\eqref{eq:Hk-bdy-est} imply
\[
    (h-k)\abs{\Omega_h}^{\frac{1}{2^*}} \leconst H(k)
\]
and
\[
    (h-k)\sigma(\Gamma(h))^{\frac{1}{s}} \leconst H(k).
\]
Defining $\phi(t) = \abs{\Omega_t} + \sigma(\Gamma_t)^{\frac{2^*}{s}}$ for all $t\ge 0$, the previous estimates imply that 
\begin{align*}
    \phi(h)&\leconst (h-k)^{-2^*} H(k)^{2^*} \\
        &\leconst (h-k)^{-2^*}\Bigl(\abs{\Omega_k}^{1-\frac{1}{p}-\frac{1}{2^*}} + \sigma(\Gamma_k)^{1-\frac{1}{q}-\frac{1}{s}}\Bigr)^{2^*} \\
        &\leconst (h-k)^{-2^*}\Bigl(\phi(k)^{1-\frac{1}{p}-\frac{1}{2^*}} + \phi(k)^{(1-\frac{1}{q}-\frac{1}{s})\frac{s}{2^*}}\Bigr)^{2^*}. 
\end{align*}
Then  
\[
    \delta := \min\bigl\{(1-\tfrac{1}{p}-\tfrac{1}{2^*})2^*,(1-\tfrac{1}{q}-\tfrac{1}{s})s\bigr\} > 1
\]
and with $p_0 = \bigl((1-\frac{1}{p}-\frac{1}{2^*})2^*-\delta\bigr)\frac{1}{2^*}\ge 0$ and $q_0 = \bigl((1-\frac{1}{q}-\frac{1}{s})s-\delta\bigr)\frac{1}{2^*}\ge 0$ one has
\begin{align*}
    \phi(h) &\leconst (h-k)^{-2^*}\phi(k)^\delta\bigl(\phi(k)^{p_0} + \phi(k)^{q_0}\bigr)^{2^*} \\
        &\leconst (h-k)^{-2^*}\phi(k)^\delta\bigl(\phi(0)^{p_0} + \phi(0)^{q_0}\bigr)^{2^*} \\
        &\leconst (h-k)^{-2^*}\phi(k)^\delta.
\end{align*}
Lemma~\ref{lem:Sta-stop-lem} implies that $\phi(t_0)=0$ for some $t_0>0$. Hence $\abs{\Omega_{t_0}}=0$ and therefore $\abs{u(x)}\le t_0$ for a.e.~$x\in\Omega$.
Now the closed graph theorem implies that~\eqref{eq:robin-bdd-est} holds for some constant $c>0$.
\end{proof}

For real $\beta$, the following elliptic regularity result was proved in a much more general form by Nittka by reflecting the solution locally at the boundary along the Lipschitz graph and using the De~Giorgi--Nash theorem. So the argument employs inner regularity of the extended solution to obtain regularity up to the boundary. More general versions of this result (with coefficients, for Hölder regularity and with more precise estimates) can be found in~\cite[Theorem~3.14]{Nit11} or~\cite[Proposition~3.3.2]{Nittka2010:thesis}, and also for mixed boundary value conditions based on a different technique in~\cite{GR01}.

\begin{prop}\label{prop:robin-reg}
Let $p>\frac{d}{2}$, $q>(d-1)$, $\lambda\ge \omega_0$ and $\beta\in L^\infty(\Gamma)$.
Let $u\in H^1(\Omega)$ be the unique solution of~\eqref{eq:robin-reg} for $f\in L^p(\Omega)$ and $g\in L^q(\Gamma)$.
Then $u\in C(\clos{\Omega})$.
\end{prop}
\begin{proof}
While in~\cite{Nit11,Nittka2010:thesis} real-valued coefficients are considered throughout, our case of complex-valued $\beta\in L^\infty(\Gamma)$ can be deduced from~\cite[Theorem~3.14]{Nit11} in the real-valued case by applying the results to the real and imaginary part of the solution separately and using the $L^\infty$-estimate in Proposition~\ref{prop:robin-bdd} which allows to handle the additional terms of the type $\Im\beta\Im u_\Gamma$ or $\Re\beta\Re u_\Gamma$ on the right hand side.
\end{proof}

\begin{rem}
Alternatively, inspection of the proof of~\cite[Theorem~3.14]{Nit11} shows that the result there stays valid in full generality for $\beta\in L^\infty(\Gamma)$ complex-valued. It suffices to observe that the result is obtained from a corresponding result for the Neumann case $\beta=0$ in~\cite[Lemma~3.10]{Nit11} and a bootstrapping argument in~\cite[Lemmas~3.11 and~3.13]{Nit11}.
The result for the Neumann case is not affected since then all coefficients are real and one can consider real- and imaginary parts separately. However, the bootstrapping argument is not affected by $\beta$ being complex-valued, since it only relies on Sobolev embedding theorems (both on $\Omega$ and for the trace as in~\eqref{eq:trace-Ls-est}) and interpolation.
\end{rem}

\begin{lem}\label{lem:C-m-diss}
Adopt the notation and assumptions of Theorem~\ref{thm:Wentzell-C}.
The operator $\Delta^\rmW_\rmC - \omega_0$ is m-dissipative.
\end{lem}
\begin{proof}
Set $A=-\Delta^\rmW+\omega_0$.
The proof of Proposition~\ref{prop:W-Lap-contr} shows that we can apply Corollary~\ref{cor:suff-inv} in order to obtain that $-A_\infty$ is m-dissipative,
where $A_\infty$ is the part of $A$ in $L^\infty(\Omega)\oplus L^\infty(\Gamma)$.

Note that
\begin{align*}
    D(A) &= \{ (u,u_\Gamma) : u\in H^1(\Omega),\ \Delta u\in L^2(\Omega),\ \partial_\nor u\in L^2(\Gamma)\},\\
    A(u,u_\Gamma) &= (-\Delta u, \partial_\nor u + \beta u_\Gamma) + \omega_0(u,u_\Gamma).
\end{align*}
Set $F = \{ (u,\restrict{u}{\Gamma}) : u\in C(\clos{\Omega})\}$. Let $\lambda>0$, $(f,\restrict{f}{\Gamma})\in F$ and
\[
    (u,u_\Gamma) = (\lambda+A_\infty)^{-1}(f,\restrict{f}{\Gamma}).
\]
Then $u\in H^1(\Omega)\cap L^\infty(\Omega)$, $\lambda u - \Delta u + \omega_0 u=f$ and $\lambda u_\Gamma + \partial_\nor u + \beta u_\Gamma + \omega_0 u_\Gamma = \restrict{f}{\Gamma}$.
It follows from Proposition~\ref{prop:robin-reg} that $u\in C(\clos{\Omega})$. Thus $u_\Gamma = \restrict{u}{\Gamma}$ and $(u,\restrict{u}{\Gamma})\in F$.
This shows that $(\lambda + A_\infty)^{-1}F\subset F$ for all $\lambda>0$.
Let $A_F$ be the part of $A$ in $F$. Then
\begin{align*}
    D(A_F) &= \{ (u,\restrict{u}{\Gamma}) \in F\cap D(A) : A(u,\restrict{u}{\Gamma})\in F\} \\
    &= \{ (u,\restrict{u}{\Gamma}) : u\in C(\clos{\Omega})\cap H^1(\Omega),\ \Delta u\in C(\clos{\Omega}),\ \partial_\nor u + \beta\restrict{u}{\Gamma}\in C(\Gamma),\\
    &\qquad  (-\Delta u + \omega_0 u, \partial_\nor u + \beta\restrict{u}{\Gamma} + \omega_0\restrict{u}{\Gamma})\in F\} \\
    &= \{ (u,\restrict{u}{\Gamma}) : u\in C(\clos{\Omega})\cap H^1(\Omega),\ \Delta u\in C(\clos{\Omega}),\ \partial_\nor u + \beta\restrict{u}{\Gamma} + \restrict{(\Delta u)}{\Gamma} = 0\}
\end{align*}
and $A_F(u,\restrict{u}{\Gamma}) = (-\Delta u+\omega_0u,\partial_\nor u + \beta\restrict{u}{\Gamma}+\omega_0\restrict{u}{\Gamma})$.
We may identify $C(\clos{\Omega})$ with $F$ via $u\mapsto (u,\restrict{u}{\Gamma})$ and
obtain $A_F=-\Delta^\rmW_\rmC +\omega_0$. Since $-A_\infty$ is m-dissipative, also $-A_F$ is m-dissipative and the claim follows.
\end{proof}

Since $\beta\in L^\infty(\Gamma)$ may not be continuous, it is not easy to see why $\Delta^\rmW_\rmC$ is densely defined in $C(\clos{\Omega})$ in general.
Following an idea of Nittka in~\cite[Lemma~4.6]{Nit11} we give a proof using the estimate in Proposition~\ref{prop:robin-bdd}. We simplify the argument for the case considered here.
\begin{lem}\label{lem:C-densedef}
Adopt the notation and assumptions of Theorem~\ref{thm:Wentzell-C}.
The operator $\Delta^\rmW_\rmC$ is densely defined in $C(\clos{\Omega})$.
\end{lem}
\begin{proof}
\begin{parenum}
\item By the Stone--Weierstrass theorem, the space $C^\infty(\clos{\Omega})$ is dense in $C(\clos{\Omega})$. So it suffices to show that every $w\in C^\infty(\clos{\Omega})$ can be approximated arbitrarily well in $C(\clos{\Omega})$ by a $u\in D(\Delta^\rmW_\rmC)$.
\item Let $w\in C^\infty(\clos{\Omega})$.
Let $\lambda\ge \omega_0$. Then~\eqref{eq:robin-reg} has a unique solution for all $(f,g)\in L^2(\Omega)\oplus L^2(\Gamma)$.
Suppose that $\Phi\in C^\infty(\clos{\Omega})$. We will choose $\Phi$ suitably later.
There exists a unique $u\in H^1(\Omega)$ such that
\begin{equation}\label{eq:Phi-rhs}
    \left\{\begin{aligned}
    \lambda u - \Delta u &= \Phi & \text{on $\Omega$,} \\
    \partial_\nor u + \beta u_\Gamma + \lambda u_\Gamma &= \restrict{\Phi}{\Gamma} & \text{on $\Gamma$.}
    \end{aligned}\right.
\end{equation}
Then $u\in D(\Delta^\rmW_\rmC)$ by Proposition~\ref{prop:robin-reg} and $\restrict{(\Delta u)}{\Gamma}=-\restrict{\Phi}{\Gamma}+\lambda\restrict{u}{\Gamma} = -\partial_\nor u - \beta\restrict{u}{\Gamma}$.

Let $\eps>0$. We show that we can choose $\Phi\in C^\infty(\clos{\Omega})$ such that $\norm{w-u}_{C(\clos{\Omega})}\le \eps$. In fact, by~\eqref{eq:Phi-rhs} we have
\[
    \lambda(u-w)-\Delta(u-w) = \Phi - \lambda w + \Delta w =: f
\]
and
\[
    \partial_\nor(u-w) + \beta\restrict{(u-w)}{\Gamma} + \lambda\restrict{(u-w)}{\Gamma} = \restrict{\Phi}{\Gamma} - \partial_\nor w - (\beta+\lambda)\restrict{w}{\Gamma} =: g.
\]
Fix $p,q>\max\{2,d\}$. Then $f\in L^p(\Omega)$ and $g\in L^q(\Gamma)$.
By Proposition~\ref{prop:robin-bdd} there exists a constant $c>0$ independent of $f$ and $g$ such that
\begin{equation}\label{eq:est-dense-diff}
    \norm{u-w}_{C(\clos{\Omega})} \le c\bigl(\norm{f}_{L^p(\Omega)} + \norm{g}_{L^q(\Gamma)}\bigr).
\end{equation}
Let $\tilde{\eps}=\frac{\eps}{c}$. Choose first $\Phi_0\in C^\infty_\cpt(\Omega)$ such that
\[
    \norm{\Phi_0 - \lambda w + \Delta w}_{L^p(\Omega)}< \frac{\tilde{\eps}}{3}
\]
and then choose $\Phi_1\in C^\infty(\RR^d)$ supported in a neighbourhood of $\Gamma$ such that both $\norm{\Phi_1}_{L^p(\Omega)}<\frac{\tilde{\eps}}{3}$ and
\[
    \norm[\big]{\restrict{\Phi_1}{\Gamma} - \partial_\nor w - (\beta+\lambda)\restrict{w}{\Gamma}}_{L^q(\Gamma)} < \frac{\tilde{\eps}}{3}.
\]
Now choose $\Phi = \Phi_0+\Phi_1\in C^\infty(\clos{\Omega})$. 
Then $\norm{f}_{L^p(\Omega)}\le2\frac{\tilde{\eps}}{3}$ and $\norm{g}_{L^q(\Gamma)}\le\frac{\tilde{\eps}}{3}$. It follows from~\eqref{eq:est-dense-diff} that
\[
    \norm{u-w}_{C(\clos{\Omega})} \le c\tilde{\eps} = \eps.
\]
This proves density.\qedhere
\end{parenum}
\end{proof}
Now Theorem~\ref{thm:Wentzell-C} follows from Lemmas~\ref{lem:C-m-diss} and~\ref{lem:C-densedef}.

Finally, we show that the semigroup generated by the Wentzell Laplacian on $L^2(\Omega)\oplus L^2(\Gamma)$ has a continuous kernel.
In order to obtain a more succinct formulation, we equip $\clos{\Omega}=\Omega\sqcup\Gamma$ with the measure $\mu$ that is the sum of the $d$-dimensional Lebesgue measure on $\Omega$ and $\Hm$ on $\Gamma$. Then $\mu$ is a finite Borel measure on the compact Euclidean space $\clos{\Omega}$, and for all $p\in[1,\infty]$ we can identify $L^p(\Omega)\oplus L^p(\Gamma)=L^p(\clos{\Omega},\mu)$. 
\begin{thm}
Adopt the notation and assumptions of Theorem~\ref{thm:Wentzell-C}.
Then the $C_0$-semigroup $(S(t))_{t\ge 0}$ on $L^2(\clos{\Omega},\mu)$ generated by $\Delta^\rmW$ has a continuous kernel on $\clos{\Omega}$, i.e.\ for all $t>0$ there exists a continuous function $k_t\colon \clos{\Omega}\times\clos{\Omega}\to\CC$ such that
\[
    \bigl(S(t)f\bigr)(x)=\int_{\clos{\Omega}} k_t(x,y)f(y)\dx[\mu(y)] 
\]
for all $f\in L^2(\clos{\Omega},\mu)$ and $x\in\clos{\Omega}$. In particular, the $C_0$-semigroup $(S_\rmC(t))_{t\ge 0}$ on $C(\clos{\Omega})$ generated by $\Delta^\rmW_\rmC$ is given by
\[
    \bigl(S_\rmC(t)f\bigr)(x)=\int_\Omega k_t(x,y)f(y)\dx[y] + \int_\Gamma k_t(x,z)f(z)\dx[\Hm(z)] 
\]
for all $t>0$, $f\in C(\clos{\Omega})$ and $x\in\clos{\Omega}$.
\end{thm}
\begin{proof}

With $\omega_0\ge 0$ as above the holomorphic $C_0$-semigroup $(T(t))_{t\ge 0}$ on $L^2(\clos{\Omega},\mu)$ given by $T(t)=e^{-\omega_0t}S(t)$ is contractive on both $L^2(\clos{\Omega},\mu)$ and $L^\infty(\clos{\Omega},\mu)$.

Fix $p_0>d$. Interpolation between $L^2(\clos{\Omega},\mu)$ and $L^\infty(\clos{\Omega},\mu)$ immediately implies that the restriction of $(T(t))_{t\ge 0}$ to $L^{p_0}(\clos{\Omega},\mu)$ is a $C_0$-semigroup on $L^{p_0}(\clos{\Omega},\mu)$. By the Stein interpolation theorem~\cite[Theorem~10.8]{AV2015:isem} we obtain that the restriction is even a holomorphic $C_0$-semigroup on $L^{p_0}(\clos{\Omega},\mu)$.

Let $-A$ denote the generator of $(T(t))_{t\ge 0}$.
Suppose that $(f,g)\in L^{p_0}(\Omega)\oplus L^{p_0}(\Gamma)$ and let $\lambda>0$.
Then it follows that $(\lambda+A)^{-1}(f,g)=(u,u_\Gamma)$ for an $u\in H^1(\Omega)$. By Proposition~\ref{prop:robin-reg} one has $u\in C(\clos{\Omega})$.
This shows that the domain of the generator of the restriction of $(T(t))_{t\ge0}$ to $L^{p_0}(\clos{\Omega},\mu)$ is contained in $C(\clos{\Omega})$.
So by holomorphy $T(t)L^{p_0}(\clos{\Omega},\mu)\subset C(\clos{\Omega})$.
Replacing $\beta$ with $\conj{\beta}$ we obtain in the same way that $T(t)^*L^{p_0}(\clos{\Omega},\mu)\subset C(\clos{\Omega})$.
Now the claim follows by~\cite[Theorem~2.1]{AtE20} and rescaling.
\end{proof}

\section{The Wentzell Laplacian for \texorpdfstring{$H^1(\Omega)$}{H¹(Ω)} functions with trace}\label{sec:Wentzell-H}

Adopt the notation and assumptions as in the beginning of Section~\ref{sec:W-Lap-apptrace}.
We consider the closed subspace 
\[
    V=\{(u,\phi)\in H^1(\Omega)\oplus L^2(\Gamma,\sigma) : \phi\in\Tr(u) \}
\]
of $H^1(\Omega)\oplus L^2(\Gamma,\sigma)$.
Let $\Delta^\rmW_V$ be the part of $\Delta^\rmW$ in $V$, i.e.
\begin{align*}
    D(\Delta^\rmW_V) &= \{ (u,\phi)\in H^1(\Omega)\oplus L^2(\Gamma,\sigma) : \Delta u\in H^1(\Omega),\ \partial_\nor u\in L^2(\Gamma,\sigma),\\
        &\qquad\qquad \phi\in\Tr(u),\ -\partial_\nor u - \beta\phi\in\Tr(\Delta u)\},\\
    \Delta^\rmW_V(u,\phi) &= (\Delta u, -\partial_\nor u - \beta \phi).
\end{align*}
We show that this realisation of the Wentzell Laplacian is a semigroup generator.

\begin{thm}\label{thm:gen-V}
Adopt the notation an assumptions of Theorem~\ref{thm:gen-Wentzell-H}.
Then the operator $\Delta^\rmW_V$ generates a holomorphic $C_0$-semigroup on $V$.
\end{thm}
\begin{proof}
It is easily verified that for $\lambda\in\RR$ sufficiently large the continuous form $\fb\colon V\times V\to\CC$ given by
\[
    \fb\bigl((u,\phi), (v,\psi)\bigr)=\int_\Omega\nabla u\cdot\conj{\nabla v} + \int_\Gamma \beta\phi\conj{\psi}\dx[\sigma] + \lambda\Bigl(\scalar{u}{v}_{L^2(\Omega)} + \scalar{\phi}{\psi}_{L^2(\Gamma,\sigma)}\Bigr)
\]
is coercive and densely defined in $H=L^2(\Omega)\oplus L^2(\Gamma,\sigma)$.
One readily verifies that it is associated with the m-sectorial operator $\lambda-\Delta^\rmW$ in $H$.

As explained in~\cite[Section~5.5.2]{Are04:evo-survey} it follows from a result by Tanabe~\cite[Section~3.6]{Tan79} that $\lambda-\Delta^\rmW_V$, the part of $\lambda-\Delta^\rmW$ in $V$, generates a holomorphic $C_0$-semigroup on $V$.
\end{proof}

If $\Omega$ is Lipschitz and $\sigma=\Hm\niv\Gamma$, then $V$ and $H^1(\Omega)$ are isomorphic. 
In fact, in this case the injective map $\tilde{j}\colon H^1(\Omega)\to V$, $\tilde{j}(u)=(u,u_\Gamma)$ is continuous by~\eqref{eq:trace-Ls-est} and surjective, and hence an isomorphism by the closed graph theorem.
Therefore we obtain the following as an immediate consequence of Theorem~\ref{thm:gen-V}.
\begin{cor}
Adopt the notation and assumptions of Theorem~\ref{thm:Wentzell-C}.
Then the operator $\Delta^\rmW_{H^1}$ given by
\begin{align*}
    D(\Delta^\rmW_{H^1}) &= \{ u\in H^1(\Omega) : \Delta u\in H^1(\Omega), \partial_\nor u + \beta u_\Gamma + (\Delta u)_\Gamma=0\},\\
    \Delta^\rmW_{H^1}u &= \Delta u,
\end{align*}
generates a holomorphic $C_0$-semigroup on $H^1(\Omega)$.
\end{cor}

\section{Appendix}

Let $(S(t))_{t\ge 0}$ be a semigroup on $L^2(X)$ for a finite measure space $(X,\Sigma,\mu)$.
Then $(S(t))_{t\ge 0}$ is called \emphdef{positive} if $f\ge 0$ implies $S(t)f\ge 0$ for all $t> 0$. Here $f\ge0$ means that $f(x)\in\RR_+$ a.e., where $\RR_+ = \{t\in\CC : \text{$\Re t\ge 0$ and $\Im t=0$}\}$.
Moreover, $(S(t))_{t\ge 0}$ is called \emphdef{submarkovian} if the semigroup is both positive and $L^\infty$-contractive.
It is easy to show that $(S(t))_{t\ge 0}$ is submarkovian if and only if it is positive and $S(t)\one_\Omega\le \one_\Omega$ for all $t>0$; see~\cite[Remark~9.3\,(d)]{AV2015:isem}.
In the following example it is convenient to show $L^\infty$-contractivity by proving that the semigroup is even submarkovian.

Let $\Omega\subset\RR^d$ be open and bounded. We define the Neumann Laplacian $\Delta^\rmN$ in $L^2(\Omega)$ by
\begin{align*}
    D(\Delta^\rmN) &= \{ u\in H^1(\Omega) : \Delta u\in L^2(\Omega),\ \partial_\nor u = 0\}, \\
    \Delta^\rmN u &= \Delta u.
\end{align*}

The next fact is well-known, of course. For the reader's convenience, we give a simple argument using the tools from Section~\ref{sec:fm-incompl}.
\begin{lem}\label{lem:N-Lap-submarkovian}
The operator $\Delta^\rmN$ generates a submarkovian $C_0$-semigroup $(S(t))_{t\ge 0}$ on $L^2(\Omega)$.
\end{lem}
\begin{proof}
\begin{parenum}
\item 
Let $D(\fa)=H^1(\Omega)$, $j\colon D(\fa)\to L^2(\Omega)$ be the inclusion and define $\fa\colon D(\fa)\times D(\fa)\to\CC$ by $\fa(u,v) = \int_\Omega\nabla u\cdot\conj{\nabla v}$.
Then $\fa(u,u) + \norm{j(u)}^2_{L^2(\Omega)} = \norm{u}^2_{H^1(\Omega)}$.
It is now straightforward to verify that $(\fa,j)$ is associated with $-\Delta^\rmN$. It follows that $\Delta^\rmN$ generates a (holomorphic) $C_0$-semigroup $(S(t))_{t\ge 0}$ on $L^2(\Omega)$.

\item The map $f\mapsto(\Re f)^+$ is the orthogonal projection of $L^2(\Omega)$ onto 
the positive cone $L^2(\Omega)_+ = \{ f\in L^2(\Omega;\RR) : f\ge 0\}$.
Since for $u\in H^1(\Omega)$ one has $(\Re u)^+\in H^1(\Omega)$,
it easily follows from Proposition~\ref{prop:suff-inv} that $S(t)L^2(\Omega;\RR)_+\subset L^2(\Omega;\RR)_+$. Hence $(S(t))_{t\ge 0}$ is positive.

\item Since $\fa(\one_\Omega,v)=0$ for all $v\in H^1(\Omega)$, one has $\one_\Omega\in D(\Delta^\rmN)$ and $\Delta^N\one_\Omega = 0$. Thus $S(t)\one_\Omega = \one_\Omega$ for all $t>0$. This implies that the semigroup is submarkovian.\qedhere
\end{parenum}
\end{proof}

Since the form in the proof of Lemma~\ref{lem:N-Lap-submarkovian} is actually $j$-elliptic in the sense of~\cite[Section~2]{AtE12:sect-form}, the sufficient condition in Proposition~\ref{prop:suff-inv} is actually necessary, see~\cite[Proposition~2.9]{AtE12:sect-form}.
This allows to give an easy abstract proof of Lemma~\ref{lem:ouh-Q-ineq}.

\begin{proof}[{Proof of Lemma~\ref{lem:ouh-Q-ineq}}]
Observe that $Q$ is the orthogonal projection of $L^2(\Omega)$ onto the closed unit ball of $L^\infty(\Omega)$. Since the $C_0$-semigroup generated by the Neumann Laplacian is $L^\infty$-contractive, \eqref{eq:Q-ineq} follows from~\cite[Proposition~2.9]{AtE12:sect-form}.
\end{proof}

\smallskip
\noindent\textbf{Acknowledgements.} This paper was inspired by the wonderful conference “PDEs and Semigroups in Applied Analysis” in honour of J.~Goldstein and S.~Romanelli which took place in Bari, July 2021. We would also like to thank the anonymous referee for their valuable input that led us to extend the exposition in Section~\ref{sec:W-Lap-apptrace} to allow fractal domains and to investigate continuity of the kernel of the semigroup for a Lipschitz domain.

\newcommand{\etalchar}[1]{$^{#1}$}

\end{document}